\documentclass[12pt]{amsart}
\usepackage[margin=1in]{geometry}

\usepackage{amsmath}
\usepackage{amssymb}
\usepackage{hyperref}
\usepackage{xcolor}
\usepackage{colonequals}

\hypersetup{colorlinks=true, linktoc=all, linkcolor=black, citecolor=blue}

\title{Counting metacyclic fields}
\author[M. Allen, J. Dell, M. Fakhari, K. J. McGown, C. Stewart, D. Tedeschi]{Maddie Allen, Justine Dell, Milad Fakhari, Kevin J.~McGown, Chloe Stewart, Daniel Tedeschi}
\date{June 2025}

\newtheorem{theorem}{Theorem}
\newtheorem{lemma}{Lemma}
\newtheorem{prop}{Proposition}

\def\Z{\mathbb{Z}}
\def\Q{\mathbb{Q}}

\DeclareMathOperator{\Gal}{Gal}

\begin{document}

\begin{abstract}
By a metacyclic field $K$ we mean the Galois closure of a pure field
$\Q(\sqrt[\ell]{D})$, $D\in\Z$, of odd prime degree $\ell$.
Let $\mathcal{M}_\ell$ denote the collection of
isomorphism classes of metacyclic fields of degree $\ell(\ell-1)$.
Write $N_\ell(X)=\#\{K\in\mathcal{M}_\ell: |\Delta_K|\leq X\}$
for the associated counting function,
where $\Delta_K$ denotes the discriminant of $K$.
We show
$$
  N_\ell(X)\sim A_\ell X^{\frac{1}{(\ell-1)^2}}(\log X)^{\ell-2}
  \,,
$$
for an explicit constant $A_\ell$.  We express $A_\ell$ as a rational number
times a product over primes of a degree $\ell$ polynomial in $1/p$.
\end{abstract}


\maketitle

\section{Introduction}\label{Intro}

By a metacyclic field $K$ we mean the Galois closure of a pure field
$\Q(\sqrt[\ell]{D})$, $D\in\Z$, of odd prime degree $\ell$.
In this case, one has
$K=\Q(\sqrt[\ell]{D},\zeta_\ell)$ where $\zeta_\ell$ denotes a primitive $\ell$-th
root of unity.  Moreover, $[K:\Q]=\ell(\ell-1)$ and the Galois group $\Gal(K/\Q)$
is isomorphic to a semidirect product
$\Z/\ell\Z\rtimes \Z/(\ell-1)\Z$.
The pure fields are classical examples of number fields, and
moreover, the cyclic extension $K/\Q(\zeta_\ell)$ furnishes an example of a
Kummer extension; such extensions play a prominent role in algebraic number theory.

In the past few decades there has been an increased interest in the subject of counting number fields.
Recently, Benli (see~\cite{Benli}) has established a counting
theorem for the pure fields of odd prime degree;
in particular, she shows that for fixed $\ell\geq 5$,
the number of such fields with discriminant bounded by $X$
is asymptotic to
$$C_\ell\, X^{\frac{1}{\ell-1}}(\log X)^{\ell-2}\,,$$
for an explicit constant $C_\ell$.
The analogous result for $\ell=3$ was already established by Cohen and Morra (see~\cite{CM}).

To our knowledge no one has yet proved a counting theorem for metacyclic fields.
It is the aim of this paper to establish such a result.
Although the metacyclic fields are obtained from the pure fields by taking Galois closure,
one counting theorem does not immediately follow from the other.
Indeed, additional arguments are needed.
One example of this kind of phenomenon in the literature is as follows:  The $S_3$-sextic
fields are the
Galois closures of the $S_3$-cubic fields;
although the counting results are related, the counting theorem for
the latter (see~\cite{DH, BST, TT}) does not immediately imply the former (see~\cite{BW}).

Fix an odd prime $\ell$.
Let
$\mathcal{M}_\ell$ denote the collection of
(isomorphism classes of)
metacyclic fields of degree $\ell(\ell-1)$.
Write
$$
  N_\ell(X)=\#\{K\in\mathcal{M}_\ell: |\Delta_K|\leq X\}
  \,,
$$
where $\Delta_K$ denotes the discriminant of $K$.

\begin{theorem}\label{T:main}
For odd prime $\ell$, the number $N_\ell(X)$ of metacyclic fields $K$ of discriminant bounded by $X$
satisfies
$$
  N_\ell(X)\sim A_\ell X^{\frac{1}{(\ell-1)^2}}(\log X)^{\ell-2}
  \,,
$$
where
\begin{align*}
    A_\ell
    &=
    a_\ell\prod_{p} \left[ \left( 1 + \frac{\ell-1}{p} \right)\left( 1 - \frac{1}{p} \right)^{\ell-1} \right],
    \\
a_\ell
&=
\frac{(\ell-1)^{-2(\ell-2)}}{(2\ell-1)(\ell-2)!}\left(2\ell^{-\frac{\ell^2-2}{(\ell-1)^2}}+\ell^{-\frac{\ell(\ell-2)}{(\ell-1)^2}}(\ell-1)^{-1}\right).    
\end{align*}
\end{theorem}

In Section~\ref{Prelim}
we will state some preliminary algebraic results concerning metacyclic fields.
In Section~\ref{Lemmas}
we will prove three lemmas that will serve as the main ingredients for our final result.
In Section~\ref{MainTheorem} we prove Theorem~\ref{T:main} via three cases:
(1) $\ell$ divides $D$;
(2) $\ell\nmid D$ and $D^{\ell-1}\not\equiv 1\pmod{\ell^2}$;
(3) $D^{\ell-1}\equiv 1\pmod{\ell^2}$.
Of course, in the third case, the condition $\ell\nmid D$ is automatic.
The asymptotics obtained in each case are given in
Propositions~\ref{P:1},~\ref{P:2}, and~\ref{P:3} respectively.


\section{Preliminaries}\label{Prelim}

Let $\ell$ be a fixed odd prime.
We adopt the same notation as in Section~\ref{Intro}.
We may assume that $D>1$ and $D$ is free of $\ell$-th powers.
Moreover, let $F=\Q(\sqrt[\ell]{D})$ and $k=\Q(\zeta_\ell)$.
In this way $F$ is a pure field of degree $\ell$, the field $k$ is the cyclotomic field
of degree $\ell-1$, and $K=Fk$ is a metacyclic field of degree $\ell(\ell-1)$.
It is perhaps relevant to note
that even with these restrictions on $D$, each field $F$ is still represented $\ell-1$ times
(see, for example~\cite[Lemma~3.2]{BMPWW}).

We write the prime factorization $D=\ell^a q_1^{e_1}\dots q_t^{e_t}$
where $0\leq a\leq \ell-1$ and $1\leq e_j\leq \ell-1$ for $j=1,\dots, t$.
We denote by $R$ the radical of $D$.
Let $f$ be the conductor associated to the extension $K/k$ via class field theory.
A priori, this conductor is an ideal in $\mathcal{O}_k$, but in the metacyclic situation it is generated
by a single element.  In fact, the conductor takes the form
\begin{equation}\label{conductor}
  f =(1-\zeta_\ell)^e q_1\dots q_t
  \,,\quad
  e\in\{0,\,2,\,\ell+1\}
  \,,
\end{equation} 
where the $q_j$ for $j=1,\dots, t$ 
are pairwise rational primes distinct from $\ell$.
It will be relevant to note that, as ideals, $(1-\zeta_\ell)^{\ell-1}=(\ell)$.
The next two results are Theorem~1 and Theorem~2 of~\cite{Mayer}.

\begin{lemma}\label{L:1}
Notation as above.  The discriminants of $F$, $K$, $k$ are given by
$$
  \Delta_{F}=\Delta_k\, f^{\ell-1}
  \,,\quad
  \Delta_K=\Delta_k^{\ell}\, f^{(\ell-1)^2}
  \,,\quad
  \Delta_k=(-1)^{\frac{\ell-1}{2}}\ell^{\ell-2}
$$
where the associated conductor satisfies
\begin{align*}
f^{\ell-1}=\begin{cases}
\ell^2 R^{\ell-1} & D^{\ell-1}\not\equiv 1\pmod{\ell^2}\,,\\
R^{\ell-1} & D^{\ell-1}\equiv 1\pmod{\ell^2}\,.
\end{cases}
\end{align*}
\end{lemma}

We define the multiplicity of $f$, denoted by $m(f)$,
the number of pure fields $F$ (up to isomorphism)
possessing the same associated conductor $f$.
We also define quantities $u(f),v(f)$ depending on $f$
so that $t=u(f)+v(f)$ as follows:
\begin{align*}
u &= u(f)=\#\{1\leq j\leq t\mid q_j^{\ell-1}\equiv 1\pmod{\ell^2}\}\,,\\
v &= v(f)=\#\{1\leq j\leq t\mid q_j^{\ell-1}\not\equiv 1\pmod{\ell^2}\}\,.
\end{align*}

The next result due to Mayer, gives a formula for $m(f)$ (see~\cite{Mayer}).

\begin{lemma}\label{L:2}
Let $f$ be the conductor associated with a pure field of degree $\ell$.
Then
$$
m(f)=\begin{cases}
(\ell-1)^t\,, &e=\ell+1\,,\\ 
(\ell-1)^{u} a(v)\,, & e=2\,,\\
(\ell-1)^{u} a(v-1)\,, & e=0\,.
\end{cases}
$$
where
\begin{align*}
a(j)&=\frac{1}{\ell}\left((\ell-1)^j-(-1)^j\right)
\,,\quad j \geq -1\,.
\end{align*}
\end{lemma}



\section{Counting lemmas}\label{Lemmas}

For $n \geq 2$, define
\[ P(\ell; n) := \sum_{p \geq \ell \text{ prime}} \frac{1}{p^n}.\] The main lemma we require is as follows:

\begin{lemma}\label{L:sum}
We have
$$
\sum_{n \leq Y} \mu(n\ell)^2(\ell-1)^{\omega(n)}
\sim
B_\ell
Y(\log Y)^{\ell-2}
\,,
$$
$$
  B_\ell=
  e^{C_\ell}\,\frac{e^{-\gamma(\ell-1)}}{\Gamma(\ell-1)} \left( 1 + \frac{\ell-1}{\ell}\right)^{-1} \prod_{p \leq \ell-1} \left( 1 + \frac{\ell-1}{p} \right)
  \,,
$$
where $\gamma$ is the Euler--Mascheroni constant, $\omega(n)$ denotes the number of distinct prime divisors of $n$, and 
\[ C_{\ell}= (\ell-1)\left(M - \sum_{p\leq \ell-1} \frac{1}{p}\right) +\sum_{n=2}^\infty (-1)^{n+1}\frac{(\ell-1)^n P(\ell;n)}{n}\] for $M$ the Meissel--Mertens constant
\[ M:=\gamma+ \sum_p \left[\log\left(1 - \frac{1}{p} \right) + \frac{1}{p} \right].\]
\end{lemma}

\begin{proof}

Let $g(n) = \mu(n\ell)^2(\ell-1)^{\omega(n)}$. Notice that $g(\ell) = 0$ and $g(p) = \ell-1$ for a prime $p \neq \ell$. Thus, the sum over primes is
\[ \sum_{p \leq Y} g(p) = \sum_{\substack{p \leq Y\\ p \neq \ell}} (\ell-1) \sim (\ell-1)\frac{Y}{\log Y} \] by the Prime Number Theorem. Additionally, $g(n)$ is multiplicative and $g(p^k) = 0$ for $k \geq 2$. Thus, the conditions for applying Wirsing's theorem \cite[Satz 1]{Wirsing} are met, and give the asymptotic expression
\begin{equation} \label{eq: Wirsing}
    \sum_{n\leq Y}\mu(n\ell)^2(\ell-1)^{\omega(n)}\sim
\frac{Y}{\log Y}\cdot\frac{e^{-\gamma(\ell-1)}}{\Gamma(\ell-1)}
\prod_{\substack{p\leq Y\\p\neq \ell}}\left(1+\frac{\ell-1}{p}\right).
\end{equation}

We can further analyze the size of the product  appearing in (\ref{eq: Wirsing}). We have that 
        \[ \prod_{p \leq Y} \left( 1 + \frac{\ell-1}{p}\right) = \exp\left({ \sum_{p \leq Y} \log\left(1 + \frac{\ell-1}{p}\right)}\right),\] so it suffices to understand the sum
        \[\sum_{p \leq Y} \log\left(1 + \frac{\ell-1}{p} \right) = \sum_{p \leq \ell-1} \log\left(1 + \frac{\ell-1}{p} \right) +\sum_{\ell-1 < p \leq Y} \log\left(1 + \frac{\ell-1}{p} \right).\] The first sum is a finite number depending only on $\ell$.
        Using the Taylor series expansion for $\log(1 + x)$ near $x = 0$ to analyze the second sum, we obtain 
        \[ \sum_{\ell-1 < p \leq Y} \log\left(1 + \frac{\ell-1}{p}\right) = \sum_{\ell-1 < p \leq Y} \frac{\ell-1}{p} + \sum_{n=2}^\infty (-1)^{n+1} \frac{(\ell-1)^n}{n}\sum_{\ell-1 < p \leq Y} \frac{1}{p^n}.\]
        
Notice that 
\begin{align}
     P(\ell; n) 
    \leq \frac{1}{\ell^n} + \sum_{p \geq \ell+1} \frac{1}{p^n} 
    \leq \frac{1}{\ell^n} + \int_{\ell}^\infty \frac{1}{x^n} \ dx 
    = \frac{1}{\ell^n} - \frac{\ell}{(1-n)\ell^n} \label{eq: P(l, n)-bound}
    \,.
\end{align}
Using partial summation, we obtain
\begin{equation*} P(\ell; n) - \sum_{\ell -1 < p \leq Y} \frac{1}{p^n} = \sum_{p > Y} \frac{1}{p^n} \ll \int_{Y}^\infty \frac{1}{(\log t)t^n} \ dt \ll \frac{1}{(\log Y)Y^{n-1}}.\end{equation*} Using this bound, Mertens' second theorem, and the Taylor expansion of $\log(1+x)$ near $x=0$, we obtain
\begin{align*}
    \sum_{\ell-1 < p \leq Y} \log\left(1 + \frac{\ell-1}{p}\right) &= \sum_{\ell-1 < p \leq Y} \frac{\ell-1}{p} + \sum_{n=2}^\infty (-1)^{n+1} \frac{(\ell-1)^n}{n} \left(P(\ell;n) + O\left( \frac{1}{(\log Y)Y^{n-1}} \right) \right) \\
    &= (\ell-1)\left(\log \log Y +M - \sum_{p\leq \ell-1} \frac{1}{p} \right) + \sum_{n=2}^\infty (-1)^{n+1}\frac{(\ell-1)^n P(\ell;n)}{n} \\
    &\hspace{1cm}+O\left( \frac{1}{\log Y} + \frac{1}{\log Y}\sum_{n=2}^\infty \frac{(\ell-1)^n}{Y^{n-1}} \right) \\
    &= (\ell-1)\left(\log \log Y +M - \sum_{p\leq \ell-1} \frac{1}{p} \right) + \sum_{n=2}^\infty (-1)^{n+1}\frac{(\ell-1)^n P(\ell;n)}{n} \\ &\hspace{1cm}+ O\left( \frac{1}{\log Y}\right),
\end{align*}
where $M$ is the Meissel--Mertens constant. Exponentiating, we obtain
        \[ \prod_{p \leq Y} \left( 1 + \frac{\ell-1}{p}\right) \sim \prod_{p \leq \ell-1} \left( 1 + \frac{\ell-1}{p} \right)e^{C_{\ell}} e^{(\ell-1)\log \log Y} = \prod_{p \leq \ell-1} \left( 1 + \frac{\ell-1}{p} \right)e^{C_{\ell}} (\log Y)^{\ell-1}, \] where 
        \[ C_{\ell} =(\ell-1)\left(M - \sum_{p\leq \ell-1} \frac{1}{p}\right) +\sum_{n=2}^\infty (-1)^{n+1}\frac{(\ell-1)^n P(\ell;n)}{n}. \]
Notice that the sum over $n$ in the expression above indeed converges as a consequence of~(\ref{eq: P(l, n)-bound}). Plugging this expression for $\prod_{p \leq Y} \left( 1 + \frac{\ell-1}{p} \right)$ into (\ref{eq: Wirsing}), we obtain the desired result.
\end{proof}

Before proceeding, we manipulate the constant $B_\ell$ into a more convenient form.

\begin{lemma} \label{L:constant}
    We have 
    \[ B_{\ell} = \frac{1}{(\ell-2)!} \left( 1 + \frac{\ell-1}{\ell} \right)^{-1}\prod_{p} \left[ \left( 1 + \frac{\ell-1}{p} \right)\left( 1 - \frac{1}{p} \right)^{\ell-1} \right].\]
\end{lemma}
\begin{proof}
    With $C_{\ell}$ defined as in Lemma \ref{L:sum}, plugging in the value of the Meissel--Mertens constant $M$ and expanding $P(\ell; n)$ gives
    \begin{align*}
        C_{\ell}-\gamma(\ell-1) &= -(\ell-1)\sum_{p\leq \ell-1} \frac{1}{p} + (\ell-1) \sum_{p} \left[ \log \left( 1 - \frac{1}{p} \right) + \frac{1}{p} \right] + \sum_{n=2}^\infty (-1)^{n+1}\frac{(\ell-1)^n P(\ell;n)}{n} \\ 
        &= (\ell-1)\sum_{p \leq \ell-1} \log\left( 1 - \frac{1}{p} \right) + \sum_{p > \ell-1} \sum_{n \geq 2} \frac{(-1)^{n+1} (\ell-1)^n - (\ell-1)}{p^nn} \\
        &= (\ell-1)\sum_{p \leq \ell-1} \log\left( 1 - \frac{1}{p} \right) + \sum_{p > \ell-1} \sum_{n \geq 1} \frac{(-1)^{n+1} (\ell-1)^n - (\ell-1)}{p^nn} \\
        &= (\ell-1)\sum_{p \leq \ell-1} \log\left( 1 - \frac{1}{p} \right) + \sum_{p > \ell-1} \left[(\ell-1)\log\left( 1 - \frac{1}{p} \right) + \log\left(1 + \frac{\ell-1}{p} \right) \right]
    \end{align*}
    where the last equality comes from the Taylor series expansion of $\log(1+x)$ near $x=0$. Exponentiating, we obtain
    \[ e^{C_{\ell}-\gamma(\ell-1)} = \prod_{p \leq \ell-1} \left( 1 - \frac{1}{p} \right)^{\ell-1} \prod_{p > \ell-1} \left[ \left( 1 + \frac{\ell-1}{p} \right)\left( 1 - \frac{1}{p} \right)^{\ell-1} \right].\]

    Thus, we have 
    \begin{align*}
        B_\ell&=
  e^{C_\ell}\frac{e^{-\gamma(\ell-1)}}{\Gamma(\ell-1)} \left( 1 + \frac{\ell-1}{\ell}\right)^{-1} \prod_{p \leq \ell-1} \left( 1 + \frac{\ell-1}{p} \right) \\&= \frac{1}{(\ell-2)!} \left( 1 + \frac{\ell-1}{\ell} \right)^{-1}\prod_{p} \left[ \left( 1 + \frac{\ell-1}{p} \right)\left( 1 - \frac{1}{p} \right)^{\ell-1} \right].
\end{align*}
\end{proof}

Finally, we also need the following upper bound.

\begin{lemma}\label{L:sum2}
We have
$$
\sum_{n \leq Y}
\mu(n\ell)^2(\ell-1)^{u(n)}
\ll_\ell
Y(\log Y)^{1-\frac{2}{\ell}}
\,.
$$
\end{lemma}

\begin{proof}
Let $f(n) = \mu(n\ell)^2(\ell-1)^{u(n)}$. Notice that $f$ is multiplicative, that $f(\ell) = 0$, and that for a prime $p \neq \ell$, $f(p) = \ell-1$ if $p^{\ell-1} \equiv 1 \pmod{\ell^2}$ and 1 otherwise. Thus, by (1.83) of \cite{IK}, we have that 
    \[ \sum_{n \leq Y} \mu(n\ell)^2(\ell-1)^{u(n)} \leq Y \prod_{\substack{p \leq Y\\ p^{\ell-1} \equiv 1 \pmod{\ell^2}}} \left( 1 + \frac{\ell-2}{p} \right).\] 
    
    We write 
    \[ \prod_{\substack{p \leq Y\\ p^{\ell-1} \equiv 1 \bmod{\ell^2}}} \left( 1 + \frac{\ell-2}{p} \right) = \exp\left({\sum_{\substack{p \leq Y\\ p^{\ell-1} \equiv 1 \bmod{\ell^2}}} \log\left( 1 + \frac{\ell-2}{p}\right)}\right).\] Using the Taylor series expansion of $\log(1+x)$ around $x=0$, we obtain 
    \begin{align*}
        \sum_{\substack{\ell-2 < p \leq Y\\ p^{\ell-1} \equiv 1 \bmod{\ell^2}}} \log\left( 1 + \frac{\ell-2}{p}\right) &= \sum_{\substack{\ell-2 < p \leq Y\\ p^{\ell-1} \equiv 1 \bmod{\ell^2}}} \frac{\ell-2}{p} + O( 1 ) \\[1ex] &= \frac{(\ell-1)(\ell-2)}{\varphi(\ell^2)}\log\log Y + O(1) \\[1ex]
        &=\frac{\ell-2}{\ell}\log\log Y + O(1),
    \end{align*} 
    where for the second to last equality, we use Mertens' Second Theorem in arithmetic progressions along with the fact that there are $\ell-1$ solutions to $x^{\ell-1} \equiv 1 \pmod{\ell^2}$. Thus, 
    \[ \prod_{\substack{p \leq Y\\ p^{\ell-1} \equiv 1 \bmod{\ell^2}}} \left( 1 + \frac{\ell-2}{p} \right) = \prod_{\substack{p \leq \ell-2\\ p^{\ell-1} \equiv 1 \bmod{\ell^2}}} \left( 1 + \frac{\ell-2}{p} \right) e^{O(1)} (\log Y)^{(\ell-2)/\ell}, \] so we obtain the bound
    \[ \sum_{n \leq Y} \mu(np)^2(\ell-1)^{u(n)} \ll_\ell Y (\log Y)^{(\ell-2)/\ell}.\]
\end{proof}




\section{Proof of Theorem~\ref{T:main}}\label{MainTheorem}

\begin{prop}\label{P:1}
Let $N_\ell^{(1)}(X)$ denote the number of $K\in\mathcal{M}_\ell$ 
with $|\Delta_K|\leq X$ such that $\ell$~divides~$D$.
We have
$$N_\ell^{(1)}(X)\sim A_\ell^{(1)} X^{\frac{1}{(\ell-1)^2}}(\log X)^{\ell-2}$$
where
$$
  A_\ell^{(1)}=B_{\ell}\,\ell^{-\frac{2\ell^2 - 2\ell -1}{(\ell-1)^2}}(\ell-1)^{-2(\ell-2)}
  \,.
$$
\end{prop}
        
\begin{proof}
Suppose $\ell\mid D$.  Then
$R = \ell q_1 \dots q_s$ for primes $q_i \neq \ell$.
By Lemma~\ref{L:1},
\begin{align*}
f &= (1-\zeta_\ell)^{\ell+1}(q_1 \dots q_s)
\,,\\
\Delta_K &= (-1)^{(\ell-1)/2}\ell^{2\ell^2-2\ell-1}(q_1 \dots q_s)^{(\ell-1)^2}
\,.
\end{align*}
We have $|\Delta_K|\leq X$ implies
$$
  q_1\dots q_s\leq \left(\frac{X}{\ell^{2\ell^2-2\ell-1}}\right)^{\frac{1}{(\ell-1)^2}}=:Y
  \,.
$$
Thus, counting the number of metacyclic fields with $|\Delta_K|\leq X$ in this case
amounts to counting the number of squarefree numbers $n$ not divisible by $\ell$
such that $n\leq Y$ with the appropriate multiplicity.  Lemma~\ref{L:2} tells us that
in this case, the appropriate multiplicity is $(\ell-1)^{\omega(n)}$.  Thus we find
$$
N_\ell^{(1)}(X)=\sum_{n\leq Y}\mu(n\ell)^2(\ell-1)^{\omega(n)}
\,.
$$
Applying Lemmas~\ref{L:sum} and~\ref{L:constant} proves the result.
\end{proof}



\begin{prop}\label{P:2}
Let $N_\ell^{(2)}(X)$ denote the number of $K\in\mathcal{M}_\ell$ 
with $|\Delta_K|\leq X$ such that $\ell\nmid D$ and $D^{\ell-1}\not\equiv 1\pmod{\ell^2}$.
We have
$$N_\ell^{(2)}(X)\sim A_\ell^{(2)} X^{\frac{1}{(\ell-1)^2}}(\log X)^{\ell-2}$$
where
$$
  A_\ell^{(2)}=B_{\ell}\,\ell^{-\frac{\ell^2-2}{(\ell-1)^2}-1}(\ell-1)^{-2(\ell-2).}
  \,.
$$
\end{prop}

\begin{proof}
Suppose $\ell\nmid D$
and
$D^{\ell-1}\not\equiv 1\pmod{\ell^2}$.
By Lemma~\ref{L:1},
\begin{align*}
|\Delta_K| &= \ell^{\ell^2-2}(q_1 \dots q_s)^{(\ell-1)^2}
\,,
\end{align*}
and hence $|\Delta_K|\leq X$ implies
$$
  q_1\dots q_s\leq \left(\frac{X}{\ell^{\ell^2-2}}\right)^{\frac{1}{(\ell-1)^2}}=:Y
  \,.
$$
As above, counting that number of metacyclic fields with $|\Delta_K|\leq X$ in this case
amounts to counting the number of squarefree numbers $n$ not divisible by $\ell$
such that $n\leq Y$ with the multiplicity
--- given by Lemma~\ref{L:2} --- equal to $(\ell-1)^{u(n)}a(v(n))$. 
Therefore, the contribution in this case is
\begin{align*}
  N_\ell^{(2)}(X)&=\sum_{n\leq Y} \mu(n\ell)^2(\ell-1)^{u(n)}a(v(n))
\\
&=
\frac{1}{\ell}
\sum_{n\leq Y} \mu(n\ell)^2(\ell-1)^{\omega(n)}
-
\frac{1}{\ell}
\sum_{n\leq Y} \mu(n\ell)^2(\ell-1)^{u(n)}
(-1)^{v(n)}
\,.
\end{align*}
Applying Lemmas~\ref{L:sum}, \ref{L:constant}, and~\ref{L:sum2} proves the result.
\end{proof}


\begin{prop}\label{P:3}
Let $N_\ell^{(3)}(X)$ denote the number of $K\in\mathcal{M}_\ell$ 
with $|\Delta_K|\leq X$ such that $D^{\ell-1}\equiv 1\pmod{\ell^2}$.
We have
$$N_\ell^{(3)}(X)\sim A_\ell^{(3)} X^{\frac{1}{(\ell-1)^2}}(\log X)^{\ell-2}$$
where
$$
  A_\ell^{(3)}=B_{\ell}\,\ell^{-\frac{\ell(\ell-2)}{(\ell-1)^2}-1}(\ell-1)^{-2(\ell-2)-1}
  \,.
$$
\end{prop}

\begin{proof}
Suppose $D^{\ell-1}\equiv 1\pmod{\ell^2}$.  In particular, $\ell\nmid D$.
By Lemma~\ref{L:1},
\begin{align*}
|\Delta_K| &= \ell^{(\ell-2)\ell}(q_1 \dots q_s)^{(\ell-1)^2}
\,,
\end{align*}
and hence $|\Delta_K|\leq X$ implies
$$
  q_1\dots q_s\leq \left(\frac{X}{\ell^{(\ell-2)\ell}}\right)^{\frac{1}{(\ell-1)^2}}=:Y
  \,.
$$
Therefore, by Lemma~\ref{L:2}, in a similar manner as before,
we have
\begin{align*}
  N_\ell^{(3)}(X)&=\sum_{n\leq Y} \mu(n\ell)^2(\ell-1)^{u(n)}a(v(n)-1)
\\
&=
\frac{1}{\ell}
\sum_{n\leq Y} \mu(n\ell)^2(\ell-1)^{\omega(n)-1}
+
\frac{1}{\ell}
\sum_{n\leq Y} \mu(n\ell)^2(\ell-1)^{u(n)}
(-1)^{v(n)}
\end{align*}
Applying Lemmas~\ref{L:sum}, \ref{L:constant}, and~\ref{L:sum2} proves the result.
\end{proof}


\begin{proof}[Proof of Theorem \ref{T:main}]
The constant $A_{\ell}$ is
\[ A_{\ell} = A_{\ell}^{(1)} + A_{\ell}^{(2)} + A_{\ell}^{(3)}.\] 
Applying Propositions~\ref{P:1}, \ref{P:2}, \ref{P:3} gives the result.    
\end{proof}

\section*{Acknowledgement}

This paper originated with the Rethinking Number Theory workshop
during the Summer of 2025, supported by NSF grant DMS-2418528.

\bibliographystyle{plain}
\bibliography{citations1}

\end{document}